\documentclass[10pt]{amsart}

\usepackage[utf8]{inputenc} 
\usepackage{geometry} 
\usepackage{graphicx} 
\usepackage{booktabs} 
\usepackage{array} 
\usepackage{paralist} 
\usepackage{verbatim} 
\usepackage{tikz}
\usetikzlibrary{arrows}
\usetikzlibrary{decorations.markings}
\usepackage{subfig}
\usepackage{tabularx}
\usepackage{enumitem}
\usepackage{amsmath,amsthm}
\usepackage{amssymb,amsfonts}
\usepackage{fancyhdr} 
\usepackage{pgf,tikz}
\usepackage{caption}
\usepackage{tikz-cd}
\usepackage{tikzscale}
\usetikzlibrary{patterns}
\usepackage{rotating}

\usepackage{color}   
\usepackage{hyperref}
\hypersetup{
    colorlinks=true, 
    linktoc=all,     
    linkcolor=blue,  
    citecolor=blue,
    filecolor=blue,
    urlcolor=blue
}

\usepackage[T1]{fontenc}
\usepackage{csquotes}
\usepackage[style=numeric,defernumbers,backend=bibtex]{biblatex}
\usepackage{hyperref}
\usepackage{nameref}
\defbibfilter{other}{
  not type=article
  and
  not type=book
  and
  not type=collection
}

\usepackage{amssymb,amsmath,latexsym,amsthm}
\usepackage{faktor} 		

\theoremstyle{plain}                    
\newtheorem{thm}{Theorem}[section]
\newtheorem{lem}[thm]{Lemma}
\newtheorem{prop}[thm]{Proposition}
\newtheorem{cor}[thm]{Corollary}

\theoremstyle{definition}
\newtheorem{defn}[thm]{Definition}
\newtheorem{ex}[thm]{Example}
\newtheorem*{qs}{Question}
\theoremstyle{remark}
\newtheorem{rmk}[thm]{Remark}

\numberwithin{equation}{section}

\newcommand{\R}{\mathbb{R}}

\newcommand{\C}{\mathbb{C}}
\newcommand{\Z}{\mathbb{Z}}

\newcommand{\hyp}{\mathbb{H}}

\newcommand{\cp}{\mathbb{C}\mathbb{P}^1}
\newcommand{\rp}{\mathbb{R}\mathbb{P}^1}
\newcommand{\pslc}{\mathrm{PSL}_2\C}

\newcommand{\pslr}{\mathrm{PSL}_2\R}

\newcommand{\dev}{\textsf{dev}}
\newcommand{\eu}[1]{\mathcal{E}(#1)}

\DeclareMathOperator{\sign}{sign}

 {\left\lbrace\begin{array}{@{}l@{}}}%
 {\end{array}\right.}

{\begin{list}{}
         {\setlength{\leftmargin}{#1}}
         \item[]
}
{\end{list}}

\def\qr#1#2{%
      \raise1ex\hbox{$#1$}\Big/ \lower1ex\hbox{$#2$}%
}

\def\qrr#1#2{%
      \raise1ex\hbox{$#1$}\Big/\Big/ \lower1ex\hbox{$#2$}%
}

\def\ql#1#2{%
      \lower1ex\hbox{$#1$}\Big\backslash \raise1ex\hbox{$#2$}%
}

 {\left\lbrace\begin{array}{@{}l@{}}}%
 {\end{array}\right.}

\begin{document}
\title[Geometrisation of purely hyperbolic representations in $\pslr$]{GEOMETRISATION OF PURELY HYPERBOLIC REPRESENTATIONS IN $\pslr$ }
\author{GIANLUCA FARACO}
\address{Dipartimento di Matematica - Universit\`a di Parma, Parco Area delle Scienze 53/A, 43132, Parma, Italy}
\curraddr{}
\email{frcglc@unife.it}

\thanks{}

\date{February 2019}

\dedicatory{}

\begin{abstract}
Let $S$ be a surface of genus $g$ at least $2$. A representation $\rho:\pi_1 S\longrightarrow \pslr$ is said to be purely hyperbolic if its image consists only of hyperbolic elements along with the identity. We may wonder under which conditions such representations arise as the holonomy of a branched hyperbolic structure on $S$. In this work we will characterize them completely, giving necessary and sufficient conditions.
\end{abstract}

\maketitle
\tableofcontents

\section{Introduction}
\subsection{About the problem} 

\noindent A branched hyperbolic structure on a surface $S$ is a geometric structure locally modeled on the hyperbolic plane, with its group of isometries $\pslr$. Any hyperbolic structure induces in a very natural way a holonomy representation $\rho:\pi_1S\longrightarrow \pslr$, that encodes geometric data about the structure. \\
\noindent The reverse problem to recover a branched hyperbolic structure from a given representation $\rho$ is more arduous and longer. Even worse there are representations from which it is not possible to recover a branched hyperbolic structure, \emph{i.e.} there are some representations that do not arise as the holonomy of a branched hyperbolic structure. In \cite{TA}, Tan gives an explicit example of such representation, that we will report here in \ref{me}.  For this reason, we will say that a representation $\rho$ is \emph{geometrisable} if it arises as the holonomy of a branched hyperbolic structure on $S$. \\
\noindent In this work we are interested in a special class of representations, namely \emph{purely hyperbolic representations}. Of course, Fuchsian representations are purely hyperbolic, and it is well-known that each of these representations is the holonomy of a unique complete hyperbolic structure (see \cite{GO88}). On the other hand, there are purely hyperbolic representations which are not Fuchsian; and then it is natural to ask if they arise as the holonomy of a branched hyperbolic structure; \emph{i.e.} a hyperbolic structure where cone points are allowed. We may immediately rule out those purely hyperbolic representations which are also elementary. Indeed, the Euler number of an elementary representation is always zero (see \cite{GO88}). On the other hand, if a representation $\rho$ arises as the holonomy of a branched hyperbolic structure then its Euler number is always different to zero. For this reason, in the sequel, we will consider only non-elementary representations.\\

\noindent Unfortunately, not all purely hyperbolic (and non-elementary) representations are the holonomy of a branched hyperbolic structure. Indeed, they arise if they satisfy a necessary (but not sufficient) condition which we will describe in subsection \ref{ss32}. If such condition holds, $\rho$ induces a Fuchsian representation $\rho_0:\pi_1\Sigma\longrightarrow \pslr$ where $\Sigma$ is closed surface of genus lower than the genus of $S$, and a map $f:S\longrightarrow \Sigma$ such that the following equality holds: $\rho=\rho_0\circ f_*$. The nature of the map $f$ determines whether a purely hyperbolic representation $\rho$ is geometrisable by a branched hyperbolic structure or not. Precisely the main result of this paper is:\\

\noindent \textbf{Theorem \ref{mainthm}:} \emph{Let $\rho:\pi_1S\longrightarrow \pslr$ be a non-Fuchsian, purely hyperbolic and non-elementary representation. Then $\rho$ is geometrisable by a branched hyperbolic structure if and only if $\rho\big(\pi_1S\big)$ is cocompact and the map
\[ f:S\longrightarrow \Sigma=\hyp^2/\rho\big(\pi_1S\big)
\] is not homotopic to a pinch map.}\\

\noindent Even if a non-elementary and purely hyperbolic representation $\rho$ does not arise as the holonomy of a branched hyperbolic structure; it arises as the holonomy of a (possibly branched) $\cp-$structure on $S$; \emph{i.e.} a geometric structure locally modeled on the Riemann sphere $\cp$ with its group of holomorphic automorphisms $\pslc$, the curious reader may see \cite{CDF} and \cite{GKM}.\\

\noindent Coming back to our structures; the problem of recovering a branched hyperbolic structure (if possible) from other types of representations is essentially open. In \cite{MA2}, Mathews considers this problem for representations $\rho$ with almost extremal Euler number, that is $\eu\rho=\pm\big(\chi(S)+1\big)$, giving some partial results. In the forthcoming paper \cite{FA} we consider the same type of representations considered by Mathews, and we improve his result giving a complete characterisation for surfaces of genus $2$.

\subsection{Structure of the paper} This paper is organized as follows. Section \ref{s1} contains the background material we need in order to tackle the main part of this work. More precisely the second section contains the basic definitions and lemmata, together with an entire paragraph of examples of purely hyperbolic representations that arise as the holonomy of a branched hyperbolic structure.\\
\noindent In section \ref{s2} we start with the main motivating example of this work: the Tan's counterexample \ref{me}. Hence we turn to show some lemmata that, all together, lead to the main theorem with its proof. Finally we will give a direct computation of the Euler number for purely hyperbolic representations.\\

\noindent \textbf{Acknowledgements.} I would like to thank my advisor Stefano Francaviglia for introducing me to this theory and for his constant encouragement. His advice and suggestions have been highly valuable. I also would like to thank Lorenzo Ruffoni for useful comments and suggestions about this work. Finally, I would like to thank an anonymous referee(s) for many useful comments and suggestions.\\

\section{Background materials}\label{s1}

\noindent Let $S$ be a closed, connected and orientable surface of genus $g$ greater than $2$. We will denote by $\hyp^2$ the hyperbolic plane and by $\pslr$ its group of orientation preserving isometries acting by M\"obius transformations
$$\pslr \times \hyp^2 \to \hyp^2, \quad \left(\begin{array}{cc} 
a & b\\ c & d \\ \end{array} \right), z \mapsto \dfrac{az+b}{cz+d}$$

\subsection{Branched hyperbolic structures} We are going to define the main structure we are interested in, that is \emph{branched hyperbolic structures}. For our purposes we only need to define branched hyperbolic structures in dimension $2$, though the following definition has obvious generalisations to higher dimensions and also other types of geometries. The curious reader may see \cite{CHK} for further details.

\begin{defn}[Hyperbolic cone-structure] A \emph{hyperbolic cone-structure} $\sigma$ on a $2$-manifold $S$ is the datum of a triangulation of $S$ and a metric, such that
\begin{itemize}
\item[1] the link of each simplex is piecewise linear homeomorphic to a circle, and  
\item[2] the restriction of the metric to each simplex is isometric to a geodesic simplex in hyperbolic space.
\end{itemize}
\end{defn}

\noindent Hence, a $2-$dimensional hyperbolic cone-structure is a surface obtained by piecing together geodesic triangles in $\hyp^2$. The definition clearly includes open surfaces and surfaces with possibly geodesic boundary. However we remind the reader that in the third part we will only consider closed surfaces.\\

\noindent Any interior point $p$ of $S$ has a neighbourhood locally isometric to $\hyp^2$, except possibly at some vertices of the triangulation, around which the angles sum to $\theta\neq 2\pi$. Such points are called \emph{cone points}. The neighbourhood of a cone point is isometric to a wedge of angle $\theta$ in the hyperbolic plane, with sides glued (that is a cone). The angle $\theta$ is called the cone angle at $p$ and letting $\theta = 2(k+1)\pi$, we define the number $k$ \emph{the order} \textsf{ord}$(p)$ of the cone point at $p$. If $S$ has boundary then this boundary will be piecewise geodesic. There may be vertices on the boundary around which the angles sum to $\theta\neq \pi$. Such points are called \emph{corner points} and the value of $\theta$ is the corner angle. Letting $\theta = \pi(1+2s)$, then $s$ is the order of the corner points. In such a case a corner point has neighbourhood isometric to a wedge of angle $\theta$ in $\hyp^2$ (without sides glued). Singular points of $\sigma$ on $S$ are cone or corner points, whereas any other points are called \emph{regular points}. Note a cone angle may be any positive real number, in particular it can be more than $2\pi$ for interior points or greater than $\pi$ for boundary points. In the sequel we will only consider closed surfaces whose cone points have order $k\in\Bbb N$. Let us now introduce the following definition.

\begin{defn}[Branched hyperbolic structure]
A \emph{branched hyperbolic structure} is a hyperbolic cone-structure such that the order of every cone point is a positive integer.
\end{defn}

\noindent We note that a complete hyperbolic structure $\sigma_0$ on $S$ can be seen as  hyperbolic cone-structure where all points are regular. Cone points may be considered as points on which the curvature is concentrated; however topology imposes limits on the allowable cone angles in a $2-$dimensional branched hyperbolic structure which can be deduced from the Gau\ss-Bonnet theorem. Precisely we have the following result.

\begin{prop}\label{gbcon}
Let $S$ be a closed, connected and orientable surface. Any hyperbolic cone-structure $\sigma$ on $S$ with cone points of order $k_p$ satisfies the following relation
\[ \chi(S)+\sum_{p \in S} k_p <0, \text{ where } k_p=\textsf{\emph{ord}}(p).
\] Indeed, the left-hand side is $2\pi$ times the negative of the hyperbolic area of $S$.
\end{prop}

\subsection{Holonomy representation} Let $\widetilde{S}$ be the universal cover of $S$ and let $\pi:\widetilde{S}\longrightarrow S$ be the covering projection. A branched hyperbolic structure  $\sigma$ on $S$ can be lifted to a branched hyperbolic structure $\widetilde{\sigma}$ on the universal cover $\widetilde{S}$.

\begin{defn} Let $\sigma$ be a branched hyperbolic structure on $S$ and $\widetilde{\sigma}$ the lifted branched hyperbolic structure on $\widetilde{S}$. A \emph{developing map} $\dev_\sigma:\widetilde{S}\longrightarrow \hyp^2$ for $\sigma$ is a smooth orientation-preserving map, with isolated critical points and such that its restriction to any simplex on $\widetilde{S}$ is an isometry.
\end{defn}

\noindent Developing maps always exist, and are essentially unique; that is two developing maps for a given structure $\sigma$ differ by post-composition with a M\"obius transformation.\\
\noindent Roughly speaking, a developing map gives a way to read the geometry of $\sigma$ on the hyperbolic plane; since $\pslr$ is the group of orientation preserving isometries for $\hyp^2$, any element acts on the hyperbolic plane without changing the information encoded by the developed image. 

\begin{rmk}
For branched hyperbolic structures the developing map $\dev$ turns out to be a branched map. Branch points are given by cone points of the branched hyperbolic structure $\widetilde{\sigma}$ on $\widetilde{S}$. Around them the developing map fails to be a local homeomorphism and the local degree is \textsf{ord}$(p)+1$, where \textsf{ord}$(p)$  is the order of the cone point.
\end{rmk}

\noindent The developing map $\textsf{dev}:\widetilde{S}\longrightarrow \hyp^2$ of branched hyperbolic structure $\sigma$ has also an equivariance property with respect to the action of $\pi_1S$ on $\widetilde{S}$. For any element $\gamma$, the composition map $\dev\circ \gamma$ is another developing map for $\sigma$. Thus, there exists an element $g\in\pslr$ such that 
\[ g\circ\dev_\sigma =\dev_\sigma\circ \gamma
\] The map $\gamma\longmapsto g$ defines a homomorphism $\rho:\pi_1S\longrightarrow \pslr$ which is called \emph{holonomy representation}. The representation $\rho$ depends on the choice of the developing map, however different choices produce conjugated representations. Hence, it makes sense to consider the conjugacy class of $\rho$, which is usually called \emph{holonomy for the structure}.\\
\noindent Although any branched hyperbolic structure $\sigma$ on a $2-$manifold $S$ induces a holonomy representation by standard argument; the reverse problem to recover a hyperbolic geometry starting from a given representation $\rho$ is more arduous and not always possible. Indeed, we know explicit examples of representations that do not arise as the holonomy of a branched hyperbolic structure (see also \ref{me}). Hence, the following definition makes sense.

\begin{defn}
A representation $\rho:\pi_1S \longrightarrow \pslr$ is said to be \emph{geometrisable by branched hyperbolic structure} if it is the holonomy of a branched hyperbolic structure $\sigma$ on $S$. Equivalently a representation is geometrisable if there exists a possibly branched developing map $\textsf{dev}:\widetilde{S}\longrightarrow \hyp^2$ which is $\rho$-equivariant. 
\end{defn}

\subsection{Euler class of representation} Throughout this subsection, $S$ will be a closed surface of genus at least $2$ unless otherwise specified. For every representation $\rho:\pi_1S\longrightarrow \pslr$ we may naturally associate a $\rp-$bundle $\mathcal{F}_\rho$ over $S$ equipped with a flat connection. Explicitly $\mathcal{F}_\rho$ is obtained as the quotient of $\widetilde{S}\times\rp$ by the diagonal action of $\pi_1S$; \emph{i.e.} for any $\gamma\in\pi_1S$ and $(p,z)\in \widetilde{S}\times \rp$ we have $\gamma\cdot (p,z)=\big(\gamma.p, \rho(\gamma)(z)\big)$. The Euler class $e(\rho)$ of $\rho$ arises naturally as an obstruction to finding global sections of this bundle.\\

\noindent  Let $\tau$ be a topological triangulation, then a section $s_0$ can be easily found on the $0-$skeleton choosing an element of $\rp$ above every vertex. This section can be extended to a section $s_1$ over the $1-$skeleton joining the $0-$sections by paths of $\rp$ elements. Since $\pi_1(\rp)=\Bbb Z$ there are infinitely many extensions of $s_0$ up to homotopy. Over any $2-$cell $T$, the section over the $1-$skeleton defines a $\rp-$vector field along $\partial T$, hence a map $\mathfrak{s}_T:\partial T\longrightarrow \rp$ of degree $d_T$ that corresponds to the number of times the vector field spins along $\partial T$. We may assign to every $2-$cell the integer $d_T$ giving a $2-$cochain $e(\rho)\in H^2(S,\Z)$.\\
\noindent  In determining $e(\rho)$ we made different choices as the triangulation $\tau$ and the $1-$section over the $1-$skeleton. Adjustment by a $2-$coboundary corresponds to altering the amount of spin chosen along each particular edge. Hence, the cohomology class of this $2-$cochain does not depend on the choice of $1-$section. Moreover, it can be seen that this cohomology class does not depend on the cellular decomposition of our surface $S$. Thus $e(\rho)$ is a well-defined $2-$cocycle called \emph{Euler class of} $\rho$ of $\mathcal{F}_\rho$.
\noindent Since $H^2(S,\Z)\cong \Z$ we can associate to $e(\rho)$ the integer $\eu\rho$ using the Kronecker pairing. We define $\eu\rho$ as the \emph{Euler number} associated to $\rho$. 

\begin{lem}\label{L321}
The Euler number satisfies the following equality 
\[ \eu\rho=\sum_{T\in \tau} d_T.
\]
\end{lem}

\proof
Let $[S]$ be the fundamental class of $S$, that is a generator of $H_2(S,\Z)$. Now $[S]=[T_1]+\cdots+[T_n]$, because $S$ is closed, that is compact without boundary; then 
\[ \eu\rho=e(\rho)[S]=\sum_{T\in\tau} e(\rho)[T]=\sum_{T\in\tau} d_T
\] where the last equality holds by definition of $e(\rho)$.
\endproof

\noindent Now suppose $\rho$ is a geometrisable representation, that is $\rho$ is the holonomy of a branched hyperbolic structure on $S$. Let $p_1,\dots,p_n$ be the cone points of orders $k_1,\dots, k_n$, respectively. The following formula relates the Euler number of $\rho$ with the Euler characteristic and the orders of the cone points.

\begin{prop}\label{P324}
Let $\rho:\pi_1S\longrightarrow \pslr$ be a representation which is the holonomy of a branched hyperbolic structure on a closed surface $S$. Then Euler number satisfies the identity
\[ \mathcal{E}(\rho)=\pm\bigg(\chi(S)+\sum_{i=1}^n k_i\bigg)
\] where the sign depends on the orientation of $S$.
\end{prop}

\begin{proof}
Among different proofs in literature we use the following argument of Mathews \cite{MA2}. Let $\tau$ be a hyperbolic triangulation, such that every cone point is a vertex of the triangulation, so we have a simplicial decomposition of $S$ with hyperbolic triangles. There is a $\rp-$vector field $V$ on $S$ with one singularity for every vertex, edge and face of $S$. The orders of the singularities are $1+k_i$ at any vertex (remember that for regular points $k=0$), $-1$ on every edge, and $1$ on every face. By the Hopf-Poincar\'e theorem the sum of the indices of the singularities equals $\chi(S)+\sum k_i$. \\
\noindent Now perturb the vector field so that the singularities lie off the $1-$skeleton. Then, the number of times the vector field spins around a triangle $T\in\tau$ is equal to the sum of the indices of singular points of $V$ inside $T$, or its negative, depending on whether the orientation induced by $\dev$ agrees with the orientation induced by the fundamental class $[S]$. For now assume these orientations agree; otherwise all the cohomology classes must be multiplied by $-1$. Hence, the spin of $V$ around any triangle $T\in\tau$ is equal to the sum of indices of singular points of $V$ inside $T$ which is in turn equal to the degree of the map $\mathfrak{s}_T:\partial T\longrightarrow \rp$ defined above. By \ref{L321} the sum of all indices of singular points is equal to $\eu\rho$, hence
\[ \eu\rho=\pm\Big(\chi(S)+\sum_{i=1}^n k_i\Big). \qedhere
\]
\end{proof}

\subsection{Some examples} Before continuing we report here examples of branched hyperbolic structures which motivate this work. In the first one we show how to obtain a $2-$dimensional branched hyperbolic structure on a closed surface $S$ by gluing the sides of a regular polygon.

\begin{ex}\label{ex1}
Let $S$ be the surface of genus $g$ by gluing the sides of a $4g$-gon with the usual labelling $a_1b_1a_1^{-1}b_1^{-1}\dots a_gb_ga_g^{-1}b_g^{-1}$. In the hyperbolic plane there are infinitely many non-isometric regular $4g-$gons, the angles on each vertex have the same value strictly between $0$ and the Euclidean one $\frac{2g\pi}{2g+1}$. Therefore we obtain a branched hyperbolic structure with only one cone point of angle strictly between $0$ and $(4g-2)\pi$; in particular we obtain a hyperbolic cone-surface in genus $g$ with one cone point of angle $2k\pi $ for any integer $1\leq k \leq 2g-2$, i.e. hyperbolic cone-surface with one cone point of order $k$ for any integer $0\leq k \leq 2g-3$.  For instance in case of $g=2$ we have the complete hyperbolic structure coming from the regular octagon of angle $\frac{\pi}{4}$, and a hyperbolic cone-surface with a cone point of angle $4\pi$ (i.e. a cone point of order $k=1$) coming from the right-angled octagon.
\end{ex}

\noindent The following example will be shown in details in the sequel; see \ref{L453}.

\begin{ex}\label{ex2} Let $\Sigma$ be a surface of genus $2$ with a complete hyperbolic structure $\sigma_0$ with Fuchsian holonomy $\rho_0$. Let $S$ be a topological surface of genus $g\ge 3$ and $f:S\longrightarrow \Sigma$ a covering map of surfaces. If the map $f$ is a genuine covering map, the structure $\sigma_0$ can be pulled back to a complete hyperbolic structure with Fuchsian holonomy $\rho$. However, if the map $f$ branchs at some points, the pull-back turns to be a branched hyperbolic structure $\sigma$ on $S$ whose cone points correspond to ramification points of $f$. In this case, the holonomy $\rho$ of $\sigma$ is a discrete but non-faithful representation of $\pi_1S$. In particular the image of $\rho$ consists only of hyperbolic transformations along with the identity because $\rho=\rho_0\circ f_*$.
\end{ex}

\begin{ex}\label{ex3} Let $S$ be a closed surface of genus $g$ with a complete hyperbolic structure $\sigma_0$ and holonomy $\rho_0$. Consider a geodesic segment of length $l$ on $S$ and cut along it to get a new surface homotopically equivalent to $S$ with an open disc removed. Geometrically the new surface inherits the branched hyperbolic structure  coming from $S$ and has a piecewise geodesic boundary $\gamma$.
Take two copies $S_1$ and $S_2$ of the new surface and glue the resulting surfaces as in  figure \ref{gp} to get a closed surface of genus $2g$ endowed with a branched hyperbolic structure $\sigma$. The holonomy $\rho$ of $\sigma$ is given by $ \rho:\pi_1S \ast_{\langle \gamma\rangle} \pi_1S\longrightarrow \pslr$. The image of $\rho$ coincides with the image of $\rho_0$, hence the representation is discrete because its image is, but not faithful.
\end{ex}
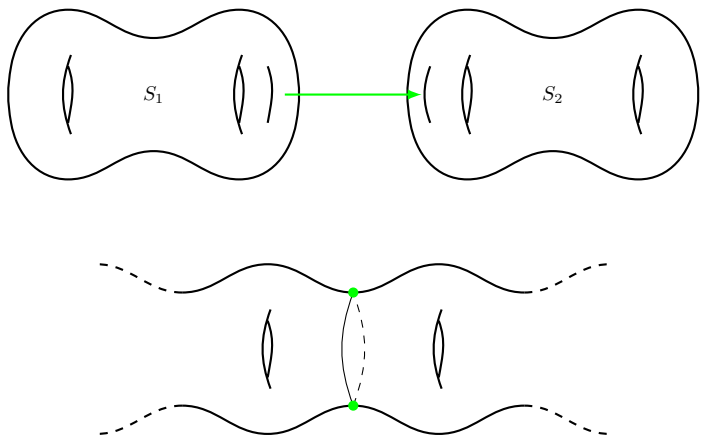
\begin{figure}[!h]
\begin{center}
\begin{tikzpicture}[thick,scale=0.75, every node/.style={scale=0.75}]

\draw [ thick] 
(3.5,6) to [out=180, in=280] (2.5,7)
(2.5,7) to [out=100, in=260] (2.5,8)
(2.5,8) to [out=80, in=180] (3.5,9)
(3.5,9) to [out=0, in=180] (5,8.5)
(5,8.5) to [out=0, in=180] (6.5,9)
(6.5,9) to [out=0, in=100] (7.5,8)
(7.5,8) to [out=280, in=80] (7.5,7)
(7.5,7) to [out=260, in=0] (6.5,6)
(6.5,6) to [out=180, in=0] (5,6.5)
(5,6.5) to [out=180, in=0] (3.5,6);

\draw [thick] 
(3.55,6.8) to [out=110, in=250] (3.55,8.2)
(3.5,7) to [out=80, in=290] (3.5,8);

\draw [thick] 
(6.55,6.8) to [out=110, in=250] (6.55,8.2)
(6.5,7) to [out=80, in=290] (6.5,8);

\draw [thick] 
(7,7) to [out=80, in=290] (7,8);

\draw [ thick] 
(10.5,9) to [out=0, in=180] (12,8.5)
(12,8.5) to [out=0, in=180] (13.5,9)
(13.5,9) to [out=0, in=100] (14.5,8)
(14.5,8) to [out=280, in=80] (14.5,7)
(14.5,7) to [out=260, in=0] (13.5,6)
(13.5,6) to [out=180, in=0] (12,6.5)
(12,6.5) to [out=180, in=0] (10.5,6)
(10.5,6) to [out=180, in=280] (9.5,7)
(9.5,7) to [out=100, in=260] (9.5,8)
(9.5,8) to [out=80, in=180] (10.5,9);

\draw [thick] 
(10.55,6.8) to [out=110, in=250] (10.55,8.2)
(10.5,7) to [out=80, in=290] (10.5,8);

\draw [thick] 
(13.55,6.8) to [out=110, in=250] (13.55,8.2)
(13.5,7) to [out=80, in=290] (13.5,8);

\draw [thick] 
(9.85,7) to [out=110, in=250] (9.85,8);

\node at (5,7.5) {$S_1$};
\node at (12,7.5) {$S_2$};

\draw [ green, latex-, thick] (9.7,7.5) to (7.3,7.5);

\draw [ thick]
(10,1.5) to [out=0, in=180] (11.5,2)
(8.5,2) to [out=0, in=180] (10,1.5)
(7,1.5) to [out=0, in=180] (8.5,2)
(5.5,2) to [out=0, in=180] (7,1.5)
(11.5,4) to [out=180, in=0] (10,4.5)
(10,4.5) to [out=180, in=0] (8.5,4)
(8.5,4) to [out=180, in=0] (7,4.5)
(7,4.5) to [out=180, in=0] (5.5,4);

\draw [thick, dashed]
(11.5,2) to [out=0, in=180] (13,1.5)
(11.5,4) to [out=0, in=180] (13,4.5)
(5.5,4) to [out=180, in=0] (4,4.5)
(5.5,2) to [out=180, in=0] (4,1.5);

\draw [thick] 
(7.05,2.3) to [out=110, in=250] (7.05,3.7)
(7,2.5) to [out=80, in=290] (7,3.5);

\draw [thick] 
(10.05,2.3) to [out=110, in=250] (10.05,3.7)
(10,2.5) to [out=80, in=290] (10,3.5);

\draw[thin,  black]
(8.5,2) to [out=110, in=250] (8.5,4);
\draw[thin, dashed,  black]
(8.5,2) to [out=70, in=290] (8.5,4);

\draw [ green] plot [mark=*, smooth] coordinates {(8.5,2)};
\draw [ green] plot [mark=*, smooth] coordinates {(8.5,4)};
\end{tikzpicture}
\end{center}
\caption[Gluing process]{We cut the surfaces along their slits and then we glue them isometrically identifying the cone-points.}
\label{gp}
\end{figure}

\section{Purely hyperbolic representations}\label{s2}

\noindent We are going to introduce a particular type of representations, namely \emph{purely hyperbolic representations.} Of course Fuchsian representations are purely hyperbolic, but also some non-Fuchsian representations are; in the previous examples \ref{ex2} and \ref{ex3} we found examples of such representations. From now on we will deal with surfaces of genus $g\ge2$. Motivated by the example \ref{ex2} we give the following definition.

\begin{defn}\label{phrdefn}
We will say that a representation $\rho:\pi_1S\longrightarrow \pslr$ is \emph{purely hyperbolic} if its image consists only of hyperbolic elements along with the identity. 
\end{defn}

\noindent We may wonder if purely hyperbolic representations arise as the holonomy of a branched hyperbolic structure. The Fuchsian case is well-known in literature, indeed Goldman's theorem \cite[Corollary D]{GO88} characterise them completely. On the other hand, in the following subsection we will give examples of purely hyperbolic representations which never arise as the holonomy of a branched hyperbolic structure.\\

\noindent Definition \ref{phrdefn} includes elementary representations. A representation $\rho$ is called \emph{elementary} if its limit set of has at most two points. In particular, the limit set of an elementary and purely hyperbolic representation has exactly two points. This implies that all elements of $\rho\big(\pi_1S\big)$ act on $\hyp^2$ leaving fixed the same axis; namely the unique axis having the points of the limit set as points at infinity. Hence, all the elements commute to one another and $\rho\big(\pi_1S\big)$ is abelian. It is well-known that abelian representations have zero Euler number (see \cite{GO88} for instance), thus no elementary and purely hyperbolic representation arises as the holonomy of a branched hyperbolic structure by \ref{gbcon}. For this reason, we can rule out elementary and purely hyperbolic representations from our discussion.\\

\noindent We recall, for the reader convenience, that a representation $\rho:\pi_1S\longrightarrow \pslr$ is said to be \emph{discrete} if its image is a discrete subgrop of $\pslr$ (with respect to the induced topology of the Lie group structure). A generic non-elementary and discrete subgroup of $\pslr$ contains hyperbolic elements, but it might contain also parabolic or elliptic elements of finite order (see \cite[Theorem 8.4.1]{B}). More precisely there is the following characterisation.

\begin{prop}\label{P32}
A non-elementary subgroup $\Gamma$ of $\pslr$ is discrete if and only if each elliptic element (if any) has finite order.
\end{prop}

\noindent By the previous proposition, we may note that any purely hyperbolic representation $\rho:\pi_1S\longrightarrow \pslr$ is discrete, \emph{i.e.} the image of $\rho$ is a discrete subgroup of $\pslr$. In \cite{GO88}, Goldman shows that faithful and discrete representations are Fuchsian. Hence non Fuchsian, purely hyperbolic representations are discrete and not faithful representations.

\subsection{Main motivating example}\label{me} The following example is a generalisation of Tan's counterexample (see \cite{TA}); which was given for surfaces of genus $3$.\\
\noindent Let $S$ be a genus $g$ surface, obtained by attaching $h$ handles to a surface of genus $g-h$, where $g-h\ge2$. We define a representation $\rho$ in the following way: $\rho$ is discrete and faithful on the original surface, and trivial on each handle we have attached. In this way $\rho(\pi_1S)$ is a discrete subgroup of $\pslr$ and the quotient $\hyp^2/\rho(\pi_1S)$ is a genus $g-h$ surface. However $\rho$ cannot be the holonomy of a branched hyperbolic structure on $S$.\\
\noindent First of all we may notice that $\eu\rho=2+2h-2g$. Suppose now that $S$ admits a branched hyperbolic structure $\sigma$ with holonomy $\rho$, and consider its developing map $\dev_\sigma:\widetilde{S}\longrightarrow \hyp^2$.
Since $\dev_\sigma$ is a $\big(\pi_1S,\rho(\pi_1S)\big)-$equivariant map; it passes down to branch map
\[ f:S\longrightarrow \ql{\rho\big(\pi_1S\big)}{\hyp^2}
\] Consider now the induced map of fundamental groups. This is the same map induced by the map that pinches to a point each handle we have attached before, hence the map $f$ is homotopic to pinching map of degree one. Since any branch cover of degree one is just a homeomorphism we found a contradiction, that is $\rho$ cannot be the holonomy of a branched hyperbolic structure. \\

\noindent So far we have examples of purely hyperbolic representations which are holonomy of a branched hyperbolic structure and examples of purely hyperbolic representations which are not. Hence the following question naturally arises.

\begin{qs}\label{q4}
Let $\rho$ be a non-Fuchsian, purely hyperbolic representation. Under which condition does $\rho$ arise as the holonomy of a branched hyperbolic structure?
\end{qs}

\subsection{A necessary condition}\label{ss32} In order to give an answer to the question \ref{q4}; a necessary condition for a purely hyperbolic representation $\rho:\pi_1S\longrightarrow \pslr$ to arise as the holonomy of a branched hyperbolic structure is the following: the quotient space $\hyp^2/\rho(\pi_1S) =\Sigma$ must be closed (hence compact without boundary); \emph{i.e.} the group $\rho\big(\pi_1S\big)$ is a cocompact subgroup of $\pslr$. More precisely we have the following lemma.

\begin{lem}\label{neccond}
Let $\rho:\pi_1S\longrightarrow \pslr$ be a purely hyperbolic representation. Suppose $\rho$ arises as the holonomy of a branched hyperbolic structure $\sigma$ on $S$; then $\rho\big(\pi_1S\big)$ is a cocompact subgroup of $\pslr$.
\end{lem}

\proof
By assumption there exists a branched hyperbolic structure $\sigma$ with holonomy $\rho$. Since $\rho$ is purely hyperbolic, the image $\rho(\pi_1S)$ is a discrete subgroup of $\pslr$ by \ref{P32}  and acts freely and properly discontinuously on the hyperbolic plane. Hence the quotient space $\hyp^2/\rho(\pi_1S) =\Sigma$ is a complete hyperbolic surface (in particular connected). It remains to show that $\Sigma$ is compact.  Let $\dev_\sigma:\widetilde{S}\longrightarrow \hyp^2$ be the developing map for $\sigma$; since it is $\big(\pi_1S,\rho(\pi_1S)\big)-$equivariant, it descends to a branched map $f:S\longrightarrow \hyp^2/\rho(\pi_1S)=\Sigma$. We note that $f$ turns out to be a proper orientation preserving map between surfaces. In particular it is a local isometry outside the branch points. According to \cite[Exercise 8.21]{SM} we claim that any proper orientation preserving map $f$ between $2-$surfaces, with at least one regular point, is surjective. Indeed this is a mapping degree matter. We may pick any regular value $q\in \Sigma$ and look at the sum
\[
\deg(f)=\sum_{f(p)=q}\sign\, d_pf,
\]
where the sign is $+1$ if $d_pf$ preserves orientation, $-1$ otherwise. Any $q\notin \textsf{Im}(f)$ is trivially a regular value and the sum is of course null. Since there is some regular value $q\in \textsf{Im}(f)$ with $\sign\, d_pf=+1$ for all $f(p)=q$ then the sum cannot be zero. Hence the conclusion; \emph{i.e.} $\Sigma$ is compact.
\endproof 

\noindent Since we are assuming $\rho$ non-elementary, by \cite[Theorem 5.2.1]{B} the image $\rho\big(\pi_1S\big)$ of $\rho$ is a Fuchsian group and the invariant set for the action of this group is the entire hyperbolic plane. 

\begin{rmk}\label{cceqfk} Let $\rho:\pi_1S\longrightarrow \pslr$ be a non-elementary and purely hyperbolic representation. We notice that $\rho\big(\pi_1S\big)$ is a cocompact subgroup if it is a Fuchsian group of the first kind, \emph{i.e.} the limit set is the entire circle at infinity. Indeed, the group $\pi_1S$ is finitely generated, hence also its image $\rho\big(\pi_1S\big)$ is. By \cite[Theorem 4.6.1]{KA} the group $\rho\big(\pi_1S\big)$ is geometrically finite; \emph{i.e.} there exists a convex fundamental region for $\rho\big(\pi_1S\big)$ with finitely many sides. 
Suppose $\rho\big(\pi_1S\big)$ is a cocompact subgroup, \emph{i.e.} $\hyp^2/\rho(\pi_1S) =\Sigma$ is compact, by \cite[Corollary 4.2.3]{KA} there is a compact fundamental region, then with finite hyperbolic area. By \cite[Corollary 4.5.2]{KA}, $\rho\big(\pi_1S\big)$ is of the first kind.
\end{rmk}

\noindent By Lemma \ref{neccond} and Remark \ref{cceqfk}, from now on we will deal only with non-elementary and purely hyperbolic representations $\rho$ such that $\rho\big(\pi_1S\big)$ is cocompact.

\subsection{Main result}\label{ss33} In this subsection we give a complete characterisation of those purely hyperbolic representations that arise as the holonomy of a branched hyperbolic structure. Some preliminaries are in order.\\
\noindent Let $\rho:\pi_1S\longrightarrow \pslr$ be a purely hyperbolic representation; by definition its image contains only hyperbolic elements along with the identity. By the discussion of the previous section \ref{ss32}, from now on we assume that $\rho(\pi_1S)$ is a purely hyperbolic cocompact subgroup of $\pslr$; this implies in particular that $\rho(\pi_1S)$ is a Fuchsian subgroup of the first kind by \ref{cceqfk}. In particular it is a discrete subgroup by proposition \ref{P32} and it acts freely and properly discontinuous on the hyperbolic plane. The following lemma holds.

\begin{lem}
The quotient space $\hyp^2/\rho(\pi_1S) =\Sigma$ is a complete hyperbolic closed surface with Fuchsian holonomy representation $\rho_0:\pi_1\Sigma\longrightarrow \pslr$. 
\end{lem}

\noindent Notice that $\rho$ and $\rho_0$ have the same image, hence there exists a map $f_*:\pi_1S\longrightarrow \pi_1\Sigma$ such that $\rho=\rho_0\circ f_*$.  Now surfaces are $K(\pi,1)$-spaces, thus any map between them is uniquely determined up to homotopy by the induced map between the fundamental groups. Thus there exists a map $f:S\longrightarrow \Sigma$ such that the induced map between fundamental groups coincides with $f_*$. That is we have shown the following proposition.

\begin{prop}\label{L033}
Let $\rho:\pi_1S\longrightarrow \pslr$ be a non-Fuchsian purely hyperbolic representation. Then there exists a closed surface $\Sigma$ of genus lower than $S$, a Fuchsian representation $\rho_0:\pi_1\Sigma\longrightarrow \pslr$ and a map $f:S\longrightarrow \Sigma$ such that
$\rho=\rho_0\circ f_*$.
\end{prop}

\noindent  We can now state the following lemma.

\begin{lem}\label{L453}
Let $f:S\longrightarrow \Sigma$ be a branched covering between surfaces. Let $\sigma_0$ be a complete hyperbolic structure on $\Sigma$ with Fuchsian holonomy $\rho_0$. Then the pull-back structure $\sigma=f^*\sigma_0$ is a branched hyperbolic structure on $S$ with purely hyperbolic holonomy $\rho$.
\end{lem}

\begin{proof}
The hyperbolic structure $\sigma_0$ may be pulled-back to a branched hyperbolic structure $\sigma$ by standard argument and cone points correspond to branch points of $f$; that is points where $f$ fails to be a local homeomorphism. The map $f$ induces a homomorphism $f_*:\pi_1S\longrightarrow \pi_1\Sigma$; and the holonomy $\rho$ for $\sigma$ is given by the composition map $\rho_0\circ f_*:\pi_1S\longrightarrow \pslr$. Hence the image of $\rho$ is contained in the Fuchsian group $\rho_0\big(\pi_1S\big)$ which is purely hyperbolic. In particular, if $\deg f\ge2$, then $\rho$ is a discrete, non-faithful representation, that is not Fuchsian.
\end{proof}

\begin{rmk}
We may note that if $f$ were a covering map in the usual sense, the same arguments show that $\rho$ is Fuchsian. Indeed in such case, by classical covering theory, the homomorphism $f_*$ turns out to be a monomorphism.
\end{rmk}

\noindent This lemma provides a sufficient condition for a purely hyperbolic representation to be holonomy of branched hyperbolic structure. Is it also necessary? We introduce the following definition.

\begin{defn}\label{pmd}
Let $f:S\longrightarrow \Sigma$ be a map between surfaces of degree $1$. We will say that $f$ is a pinch map if there are two simple closed, non-contractible, curves $\alpha$ and $\beta$ meeting transversally at a single point such that $f(\alpha)$ and $f(\beta)$ are contractible in $\Sigma$.
\end{defn}

\noindent Notice that composition of two or more pinch maps as in definition \ref{pmd} is still a pinch map. A pinch map $f:S\longrightarrow \Sigma$ is said to be \emph{simple} if the genus of $\Sigma$ is one less than the genus of $S$. In particular, any pinch map is homotopic to composition of simple pinch maps. We now state the following result.

\begin{lem}\label{L456}
Let $\rho:\pi_1S\longrightarrow \pslr$ be a non-Fuchsian and purely hyperbolic representation. Let $\rho_0:\pi_1\Sigma\longrightarrow \pslr$ be a Fuchsian representation, where the genus of $\Sigma$ is strictly lower than the genus of $S$. Suppose there is a map $f:S\longrightarrow \Sigma$ such that
\begin{itemize}
\item[1] $f$ is a pinch map,
\item[2] $f_*:\pi_1S\longrightarrow \pi_1\Sigma$ is such that $\rho=\rho_0\circ f_*$.
\end{itemize}
Then $\rho$ does not arise as the holonomy of a branched hyperbolic structure.
\end{lem}

\begin{proof}
The lemma follows from similar arguments of \ref{me}. First of all we may notice that $f_*:\pi_1S\longrightarrow \pi_1\Sigma$ is surjective because $f$ is a pinch map; thus $\rho\big(\pi_1S\big)=\rho_0\big(\pi_1\Sigma\big)$. Suppose there exists a branched hyperbolic structure $\sigma$ with holonomy $\rho$ and consider its developing map $\dev_\sigma:\widetilde{S}\longrightarrow \hyp^2$. Since it is $\big(\pi_1S,\rho(\pi_1S)\big)-$equivariant, it descends to a branched map $b:S\longrightarrow \hyp^2/\rho(\pi_1S)=\Sigma$. The induced map $b_*$ on the fundamental groups is such that  $\rho=\rho_0\circ b_*$,  thus it is just that of the pinching map because it coincides with $f_*$. Hence $\deg\big( S\longrightarrow \hyp^2/\rho(\pi_1S)\big)=1$, implying that such map is a branched map of degree one, that is a homeomorphism, a contradiction.
\end{proof}

\noindent The following corollary is immediate.

\begin{cor}\label{corpinch}
Let $f:S\longrightarrow \Sigma$ be a pinch map. Let $\sigma_0$ be a complete hyperbolic structure on $\Sigma$ with Fuchsian holonomy $\rho_0$. Then $\sigma_0$ cannot be pulled-back to a branched hyperbolic structure on $S$.
\end{cor}

\begin{lem}\label{Ppb}
Let $f:S\longrightarrow \Sigma$ be the composition of a pinch map with a branched (but not unbranched) covering map. Let $\sigma_0$ be a complete hyperbolic structure on $\Sigma$ with Fuchsian holonomy $\rho_0$. Then $\sigma_0$ can be pulled-back to a branched hyperbolic structure on $S$.
\end{lem}

\begin{proof}
By assumption, $f$ is the composition of a pinch map $\pi:S\longrightarrow \pi(S)$ with a branched covering map $b:\pi(S)\longrightarrow \Sigma$. Set $\pi(S)=T$. First of all we explain why we need to assume $b$ to be a branched, but not unbranched, covering map. If $b$ were an unbranched covering map, then the complete hyperbolic structure $\sigma_0$ pulls back to a complete hyperbolic structure $\sigma$ on $T$. By \ref{corpinch}, $\sigma$ cannot be pulled back to a branched hyperbolic structure on $S$. Hence we assume $b$ to be a branched covering map. In this case, the complete hyperbolic structure $\sigma_0$ pulls back to a branched hyperbolic structure $\sigma_T$ on $T$ with purely hyperbolic holonomy $\rho_T$. We first assume that $\pi$ is a simple pinch map, hence the genus of $T$ is one less than the genus of $S$. Set $\rho=\rho_T\circ\pi_*=\rho_0\circ f_*$. As above, the map $\pi_*:\pi_1S\longrightarrow \pi_1T$ is surjective because $\pi$ is a pinch map, thus $\rho\big(\pi_1S\big)=\rho_T\big(\pi_1T\big)$. Given the well-known standard presentation of the fundamental group of $S$, namely $\pi_1S=\big\langle a_1,b_1,\dots, a_{g},b_{g} \vert [a_1,b_1]\cdots[a_{g},b_{g}]=1\big\rangle$ where $g$ is the genus of $S$, we assume without loss of generality that $\pi_*(a_g)=\pi_*(b_g)=1$. Since $\sigma_T$ is a branched hyperbolic structure with non-Fuchsian holonomy, there is a cone point $p$. Its order $k$ is at least $2$, hence the magnitude of the angle around $p$ is $2k\pi$. Let $\tau_1$ and $\tau_2$ two geodesic segments starting from $p$ both of length $2l$ and such that
\begin{itemize}
\item[\bf 1.] $\tau_1$ and $\tau_2$ intersect only at $p$,
\item[\bf 2.] the angle between $\tau_1$ and $\tau_2$ is $2h\pi$ with $0<h<k$.
\end{itemize}
Notice that a couple of such segments can be always found. Moreover, it is an easy matter to see that $\tau_1$ and $\tau_2$ have the same developed image in $\hyp^2$. For $i=1,2$, define $p_i$ as that extremal point of $\tau_i$ different to $p$. Then define $\gamma_i$ as the sub-arc of $\tau_i$ starting from $p_i$ of length $l$. For $i=1,2$, cut along $\gamma_i$ to get a surface with a piecewise geodesic boundary $\gamma_i^1\cup\gamma_i^2$ and two corner angles. Then glue $\gamma_1^i$ with $\gamma_2^i$, as shown in the figure \ref{cp}, producing an additional handle. 

\begin{figure}[!h]
\centering
\begin{tikzpicture}[thick,scale=0.7, every node/.style={scale=0.7}]
\draw [thick, red] (2.5,5) to (4,5);
\draw [thick, red] (5.5,5) to (4,5);
\draw [thick, green] (2.5,5) to (1,5);
\draw [thick, green] (5.5,5) to (7,5);

\draw [thick] plot [mark=*, smooth] coordinates {(4,5)};
\draw [thick] plot [mark=*, smooth] coordinates {(1,5)};
\draw [thick] plot [mark=*, smooth] coordinates {(7,5)};
\draw [thick] plot [mark=*, smooth] coordinates {(2.5,5)};
\draw [thick] plot [mark=*, smooth] coordinates {(5.5,5)};

\node at (4,5.5) {\text{\emph{p}}};
\node at (1,5.5) {\text{\emph{$p_1$}}};
\node at (7,5.5) {\text{\emph{$p_2$}}};
\node at (1.75,4.5) {\text{\emph{$\gamma_1$}}};
\node at (6.25,4.5) {\text{\emph{$\gamma_2$}}};

\draw[thick, -latex] (8,5) to (10,5);

\draw [thick, red] (12.5,5) to (14,5);
\draw [thick, red] (15.5,5) to (14,5);
\draw[thick, green] (12.5,5) to[out=135, in=45] (11,5);
\draw[thick, green] (12.5,5) to[out=225, in=315] (11,5);
\draw[thick, green] (17,5) to[out=135, in=45] (15.5,5);
\draw[thick, green] (17,5) to[out=225, in=315] (15.5,5);

\draw [thick] plot [mark=*, smooth] coordinates {(14,5)};
\draw [thick] plot [mark=*, smooth] coordinates {(12.5,5)};
\draw [thick] plot [mark=*, smooth] coordinates {(11,5)};
\draw [thick] plot [mark=*, smooth] coordinates {(15.5,5)};
\draw [thick] plot [mark=*, smooth] coordinates {(17,5)};

\node at (14,5.5) {\text{\emph{p}}};
\node at (11,5.5) {\text{\emph{$p_1$}}};
\node at (17,5.5) {\text{\emph{$p_2$}}};
\node at (11.75,4.25) {\text{\emph{$\gamma_1^1$}}};
\node at (16.25,4.25) {\text{\emph{$\gamma_2^1$}}};
\node at (11.75,5.75) {\text{\emph{$\gamma_1^2$}}};
\node at (16.25,5.75) {\text{\emph{$\gamma_2^2$}}};

\draw[thick, -latex] (15,3) to (11,1);

\draw[ thick] (10,2.5) to[out=170, in=0] (7,3);
\draw[ thick] (7,3) to[out=180, in=90] (4,1);
\draw[ thick] (4,1) to[out=270, in=180] (7,-1);
\draw[ thick] (7,-1) to[out=0, in=190] (10,-0.5);

\draw [ultra thick, red] (6,1) to [out=330, in=210] (8,1);
\draw [thick] (5.8,1.2) to [out=300, in=150] (6,1);
\draw [thick] (8,1) to [out=30, in=240] (8.2,1.2);
\draw [ultra thick, red] (6,1) to [out=30, in=150] (8,1);
\draw [ultra thick, green] (6,1) to [out=150, in=30] (4,1);
\draw [ultra thick, dashed, green] (6,1) to [out=210, in=330] (4,1);

\draw[thick] plot [mark=*, smooth] coordinates {(6,1)};
\draw[thick] plot [mark=*, smooth] coordinates {(8,1)};
\draw[thick] plot [mark=*, smooth] coordinates {(4,1)};

\node at (8,1.5) {\text{\emph{p}}};
\end{tikzpicture}
\caption[]{The segment $\gamma_1$ is the sub-arc of $\tau_1$ starting from $p_1$. Similarly, the segment $\gamma_2$ is the sub-arc of $\tau_2$ starting from $p_2$. Both of them are coloured in green. Cut along $\gamma_1$ and $\gamma_2$ in order to get the situation showed in the picture on the right. Now glue $\gamma_{1}^1$ with $\gamma_{2}^1$ and $\gamma_{1}^2$ with $\gamma_2^2$. The result is the handle shown in the picture below.}
\label{cp}
\end{figure}
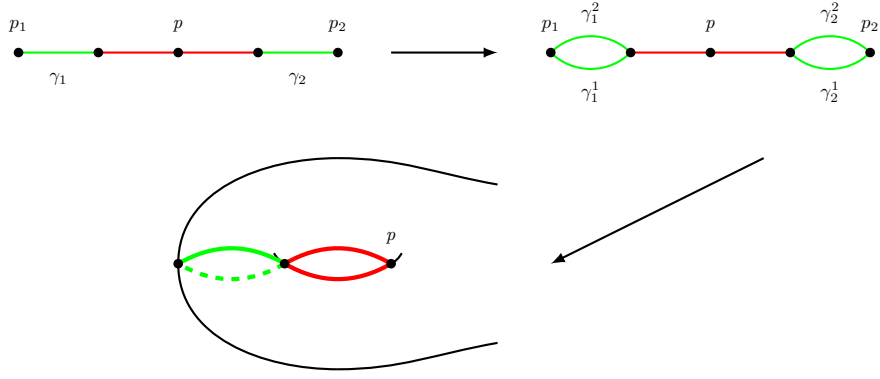

\noindent The surface we get is $S$ endowed with a branched hyperbolic structure $\sigma$ with holonomy $\overline{\rho}$. We are going to show that $\overline{\rho}=\rho$. The cut and paste procedure described above produces two new closed curves, say $a_g$ and $b_g$, such that together with $a_1,b_1,\dots, a_{g-1},b_{g-1}$ generate $\pi_1S$ and satisfies the relation among the commutators. By construction, $\overline{\rho}(a_i)=\rho_T(a_i)$ and $\overline{\rho}(b_i)=\rho_T(b_i)$ for $i=1,\dots,g-1$ and $\overline{\rho}(a_g)=\overline{\rho}(b_g)=1$. Hence the map $\pi_*$ is such that $\overline{\rho}=\rho_T\circ\pi_*$, thus $\overline{\rho}=\rho$. Finally, the case of $\pi$ is the composition of simple pinch maps follows by a recursive application of the previous argument.
\end{proof}

\noindent In order to prove the main theorem we invoke the following result.

\begin{thm}[Edmonds, \cite{EA}]\label{et}
If $f: S\longrightarrow \Sigma$ is a map of nonzero degree between closed orientable surfaces, then there is a pinch map $\pi:S\longrightarrow T$ and there is a branched covering $b:T\longrightarrow \Sigma$ such that the composition $b\circ\pi$ is homotopic to $f$.
\end{thm}

\noindent Using Edmonds' theorem together with the lemmata \ref{L453}, \ref{L456} and \ref{Ppb}, we are able to prove our main theorem.

\begin{thm}\label{mainthm}
Let $\rho:\pi_1S\longrightarrow \pslr$ be a non-Fuchsian, purely hyperbolic and non-elementary representation. Then $\rho$ is geometrisable by a branched hyperbolic structure if and only if $\rho\big(\pi_1S\big)$ is cocompact and the map
\[ f:S\longrightarrow \Sigma=\hyp^2/\rho\big(\pi_1S\big)
\] is not homotopic to a pinch map.
\end{thm}

\begin{proof}
Let $\rho:\pi_1S\longrightarrow \pslr$ be a non-Fuchsian, purely hyperbolic and non-elementary representation. By Proposition \ref{L033}, there exists a closed surface $\Sigma$ of genus lower than $S$, a Fuchsian representation $\rho_0:\pi_1\Sigma\longrightarrow \pslr$ and a map $f:S\longrightarrow \Sigma$ such that $\rho=\rho_0\circ f_*$. Consider the map $f$. By Edmonds' theorem \ref{et}; there exists an intermediate surface $T$, a pinch map $\pi:S\longrightarrow T$ and a branched covering $b:T \longrightarrow \Sigma$ such that the composition $b\circ\pi$ is homotopic to $f$. Now the sufficient condition comes from Lemmata \ref{L453} and \ref{Ppb}, whereas the necessary condition follows from Lemma \ref{L456}.
\end{proof}

\begin{rmk}
By a recent result of March\'e-Wolff in \cite{MW}, for closed surfaces of genus $2$ there are no purely hyperbolic representations which are not Fuchsian. Moreover, as we will see below in remark \ref{pven}, the Euler numbers of purely hyperbolic representations that arise as the holonomy of a branched hyperbolic structure are always even and different to zero.
\end{rmk}

\begin{rmk}\label{R314}
Let \textsf{Hom}$(\pi_1S,\pslr)$ be the representation variety of all representation $\pi_1S\longrightarrow \pslr$. This space turns out to be a disjoint union of $4g-3$ connected components; which are parametrized by the Euler number (see \cite{GO88}).
In \cite[Proposition 1.2]{FW}, the authors show that the set of discrete and non-faithful representations form a nowhere dense closed subset in each component of the representation variety . Hence, purely hyperbolic representations are essentially rare. In the following subsection we show that they do not appear in each component of the representation variety.
\end{rmk}

\subsection{Euler number of purely hyperbolic representations}
\noindent Let $\rho:\pi_1S\longrightarrow \pslr$ be a purely hyperbolic representation. As above, suppose $\rho\big(\pi_1S\big)$ is of the first kind. It is natural to ask which are the possible values of the Euler number $\eu\rho$. The following result follows from a straightforward computation.

\begin{lem}\label{L312}
Let $f:S\longrightarrow \Sigma$ be a branched covering map, and $\rho_0:\pi_1\Sigma\longrightarrow \pslr$ be a Fuchsian representation. Consider the representation $\rho=\rho_0\circ f_*:\pi_1S\longrightarrow \pslr$; then 
\[  \eu\rho=d\cdot \eu{\rho_0}
\] where $d$ is the degree of $f$.
\end{lem}

\noindent Hence, the following lemma follows immediately from the previous one.

\begin{lem}\label{L313}
Let $\rho$ be a non-Fuchsian, purely hyperbolic representation, such that $\rho\big(\pi_1S\big)$ is a cocompact subgroup of $\pslr$. Then $\eu\rho$ is even.
\end{lem}

\begin{proof}
By lemma \ref{L033} there exists a closed surface $\Sigma$ of genus lower than $S$, a Fuchsian representation $\rho_0:\pi_1\Sigma\longrightarrow \pslr$ and a map $f:S\longrightarrow \Sigma$ such that
$\rho=\rho_0\circ f_*$.  By lemma \ref{L312}, we have that $\eu\rho=d\cdot \eu{\rho_0}$; where $d$ is the degree of $f$. Since $\rho_0$ is Fuchsian, then $\eu{\rho_0}=\pm\chi(\Sigma)=\pm\big( 2-2g_\Sigma\big)$. Hence $\eu\rho$ is always even.
\end{proof}

\begin{rmk}\label{pven}
Let $\rho:\pi_1S\longrightarrow \pslr$ be a purely hyperbolic representation arising as the holonomy of a branched hyperbolic structure $\sigma$ on $S$. By \ref{L313}, the Euler number of $\rho$ is even, however it can never be zero. Indeed, suppose $\eu\rho=0$, then we have the following absurd chain of equalities and inequalities:
\[ 0=\eu\rho=\pm\bigg(\chi(S)+\sum_{i=1}^n k_i \bigg)<0
\] where $k_i$ are the orders of the conical points of $\sigma$. The second equality holds by \ref{P324}, and the inequality holds by \ref{gbcon}.
\end{rmk}

\noindent Following remark \ref{R314} we give an example of a non-purely hyperbolic representation with even Euler number.

\begin{ex}\label{ex4} Despite the Euler number of purely hyperbolic representation is always even (and not zero) by Lemma \ref{L313}, not every geometrisable non- Fuchsian representation with even Euler number is purely hyperbolic. Here we propose an explicit example. Let $S$ be a closed surface of genus $2$ with a branched hyperbolic structure $\sigma_0$ with a single cone point of angle $4\pi$ and let $\rho_0$ be its holonomy representation. Notice that $\eu{\rho_0}=\pm1$, depending on the orientation of $S$, suppose $\eu{\rho_0}=-1$. Consider a geodesic segment of length $l$ on $S$ and cut along it to get a new surface, say $\dot S$ homotopically equivalent to $S$ with an open disc removed. Geometrically speaking, the new surface $\dot S$ inherits the branched hyperbolic structure coming from $S$ and has a piecewise geodesic boundary $\gamma$ with two corner points of angle $2\pi$. As in \ref{ex3}, take two copies $S_1$ and $S_2$ of $\dot S$ and glue them along the boundaries with the obvious identification as in figure \ref{gp}. The resulting surface turns out to be a closed surface of genus $4$ endowed with a branched hyperbolic structure $\sigma$ with four cone points of angle $4\pi$. The holonomy $\rho$ of $\sigma$ is given by $ \rho:\pi_1S \ast_{\langle \gamma\rangle} \pi_1S\longrightarrow \pslr$. In \cite{MW}, the Authors showed that for surfaces of genus two any non-Fuchsian representation sends a simple curve to a non-hyperbolic element, hence $\rho_0$ does since $\eu{\rho_0}=-1$. The image of $\rho$ clearly coincides with the image of $\rho_0$, hence we may conclude that $\rho$ is not purely hyperbolic.
\end{ex}

\printbibliography
\rhead[\fancyplain{}{\bfseries
Bibliography}]{\fancyplain{}{\bfseries\thepage}}
\lhead[\fancyplain{}{\bfseries\thepage}]{\fancyplain{}{\bfseries
Bibliography}}

\end{document}